\documentclass[12pt,reqno]{article}
\usepackage{amscd,mathrsfs, amsfonts, graphicx, color}
\usepackage{amsmath}
\usepackage{amssymb}
\usepackage{amsopn}
\usepackage{amsthm}
\usepackage{amstext}
\RequirePackage{geometry}
\newtheorem{thm}{Theorem}[section]
\newtheorem{cor}[thm]{Corollary}
\newtheorem{lem}[thm]{Lemma}
\newtheorem{prop}[thm]{Proposition}
\theoremstyle{remark}
\newtheorem{defn}[thm]{\bf Definition}
\newtheorem{rem}[thm]{\bf Remark}

\numberwithin{equation}{section}

\def\Ran{\mathcal R}
\def\Ker{\mathcal N}
\def\R{\mathcal{R}}

\def\D{{\rm dist}}

\def\N{\mathcal{N}}

\begin{document}

%------------------------------------------------------------------------------------%
\title{Perturbations and expressions of the Moore--Penrose metric generalized inverses and applications to the stability of some operator equations}
%\author[J. B. Cao]{Jianbing Cao$^*$}
\author{Jianbing Cao\thanks{{\bf Email}: caocjb@163.com} \\
Department of mathematics, Henan Institute of Science and Technology\\
Xinxiang, Henan, 453003, P.R. China\\
Department of Mathematics, East China Normal University,\\
Shanghai 200241, P.R. China\\
\and
Yifeng Xue\thanks{{\bf Email}: yfxue@math.ecnu.edu.cn; Corresponding author}\\
Department of Mathematics, East China Normal University,\\
Shanghai 200241, P.R. China }
\date{}

\maketitle

%------------------------------------------------------------------------------------%
%\begin{center}
%Communicated by...\;
%\end{center}
%------------------------------------------------------------------------------------%

\begin{abstract}
In this paper, the problems of perturbation and expression for the Moore--Penrose metric generalized inverses of bounded
linear operators on Banach spaces are further studied. By means of certain geometric assumptions of Banach spaces,
we first give some equivalent conditions for the Moore--Penrose metric generalized inverse of perturbed operator to
have the simplest expression $T^M(I+ \delta TT^M)^{-1}$. Then, as an application our results, we investigate the
stability of some operator equations in Banach spaces under different type perturbations.
\vspace{3mm}

 \noindent{2010 {\it Mathematics Subject Classification\/}: Primary 47A05; Secondary 46B20}

 \noindent{{\it Key words}: Metric generalized inverses, stable perturbation, operator equations}
\end{abstract}

\vskip 0.2 true cm

%------------------------------------------------------------------------------------%

\pagestyle{myheadings}

\markboth{\rightline {\scriptsize  J. B. Cao \& Y. F. Xue}}
         {\leftline{Metric generalized inverses and operator equations}}

%\bigskip
\bigskip

%------------------------------------------------------------------------------------%
%------------------------------------------------------------------------------------%
\section{Introduction}
Throughout the paper, we always let $X$ and $Y$ be Banach spaces, and $B(X, Y)$ be the Banach space consisting of all bounded linear operators from $X$ to $Y$. For $T\in B(X,Y)$, let $\Ker(T)$ (resp. $\Ran(T)$) denote the kernel (resp. range) of
$T$. It is well--known that for $T\in B(X,Y)$, if $\Ker(T)$ and $\Ran(T)$ are topologically complemented in the spaces $X$ and $Y$, respectively, then there exists a linear projector generalized inverse $T^+\in B(Y, X)$ defined by
$$
T^+Tx=x,\ x\in \Ker(T)^c\ \text{and}\ T^+y=0,\ y\in \Ran(T)^c,
$$
where $\mathcal{N}(T)^c$ and $\mathcal{R}(T)^c$ are topologically complemented subspaces of $\mathcal{N}(T)$ and $\mathcal{R}(T)$, respectively. We know that linear projector generalized inverses of bounded linear operators have many important
applications in numerical approximation, statistics and optimization et al. (see \cite{BG1,NV1,WYW1,XYF1}). But, generally speaking, the linear projector generalized inverse can not deal with the extremal solution, or the best approximation
solution of an ill--posed operator equation in Banach spaces. In order to solve the best approximation problems for an ill--posed linear operator equation in Banach spaces, Nashed and Votruba \cite{NMZF1} introduced the concept of the (set--valued) metric generalized inverse of a linear operator in Banach spaces. Later, in 2003, H. Wang and Y. Wang \cite{WWH11} defined the Moore--Penrose metric generalized inverse for a linear operator with closed range in Banach spaces, and gave some useful characterizations. Then in \cite{NRX1}, the author defined and characterized the Moore--Penrose metric generalized inverse for an arbitrary linear operator in a Banach space. From then on, many research papers about the Moore--Penrose metric generalized inverses have appeared in the literature, see \cite{BWLX1,HWZh1,Mah1q2,WLJ1,WLP1,ZhYWW1} for instance.

In his recent thesis\cite{Mah1}, H. Ma presented some perturbation
results of the Moore--Penrose metric generalized inverses under
certain additional assumptions, and also obtained some descriptions
of the Moore--Penrose metric generalized inverses in Banach spaces.
It is well--known that the perturbation analysis of generalized
inverses of linear operators has wide applications and plays an
important role in many fields such as computation, control theory,
frame theory and nonlinear analysis. While the metric projector on
closed subspace in Banach space are no longer linear, and then the
linear projector generalized inverse and the Moore--Penrose metric
projector generalized inverse of a bounded linear operator in Banach
space are quite different. Motivated by many perturbation results of
the linear operator generalized inverses in Hilbert spaces or Banach
spaces \cite{CX1,DJ1,DJ12,XYF1} and some recent results in
\cite{Mah1}, in this paper, we further study the following
perturbation and representation problems for the Moore--Penrose
metric generalized inverses: let $T \in B(X,Y)$ such that the
Moore--Penrose metric generalized inverse $T^M$ of $T$ exists, what
conditions on the small perturbation $\delta T$ can guarantee that
the Moore--Penrose metric generalized inverse $\bar{T}^M$ of the
perturbed operator $\bar{T} = T+ \delta T$ exists? Furthermore, if
it exists, when does $\bar{T}^M$ have the simplest expression $(I_X
+ T^M\delta T)^{-1}T^M$?  Under the geometric assumption that the
Banach spaces $X$ and $Y$ are smooth, and by using the generalized
orthogonal decomposition theorem \cite{WYW1}, we will give a
complete answer to these problems. Meanwhile, by using the reduced
minimum module and the gap function, we also characterize the
existence of the Moore--Penrose metric generalized inverse of the
perturbed operator in reflexive strictly convex Banach spaces.
The obtained results extend and improve many recent results in this field, for instance
\cite[Chapter 4]{Mah1}.

Perturbation analysis for the extremal solution of the linear
operator equation $Tx = y$ by using the linear generalized inverses (see \cite{CX1,DJ22}) have been made by many authors.
It is well known that the theory of the Moore--Penrose metric generalized inverses has its genetic in the context of the
so--called ``ill--posed'' linear problems. So, as applications of our results, in the last section of this paper, we will
investigate the stability of the solutions of the operator equation $Tx = y$ and the best approximate solutions of the
operator equation $\|Tx-b\|=\inf_{y\in X}\|Ty-b\|$ in Banach spaces under some different conditions.

\section{\bf Preliminaries}
In this section, we will recall some concepts and results frequently used in this paper. We first recall some related concepts about homogeneous operators and the geometry of Banach spaces. For more information about the geometric properties of Banach spaces, such as strict convexity, reflexivity and complemented subspaces, we refer to \cite{Be,BP1,Diestel1}.

Let $X, Y$ be Banach spaces, let $T: X \to Y$ be a mapping and $D\subset X$ be a subset of $X$. Recall from \cite{BWLX1,WP11} that a subset $D$ in $X$ is called to be homogeneous if $\lambda\,x\in D$ whenever
$x\in D$ and $\lambda\in\mathbb R$; a mapping $T\colon X \rightarrow Y$ is called to be a bounded homogeneous operator
if $T$ maps every bounded set in $X$ into a bounded set in $Y$ and $T(\lambda\, x)=\lambda\, T(x)$ for every $x\in X$ and every
$\lambda\in\mathbb R$. Let $H(X,Y)$ denote the set of all bounded homogeneous operators from $X$ to $Y$. Equipped with the usual linear operations on $H(X,Y)$ and norm on $ T\in H(X,Y)$ defined by $\|T\|=\sup\{\|Tx\|\,|\, \|x\|=1, x\in X\}$, we can easily prove that $(H(X,Y), \|\cdot\|)$ is a Banach space (cf. \cite{WYW1,WP11}). For a bounded
homogeneous operator $ T\in H(X,Y)$, we always denote by $\mathcal{D}(T)$, $\mathcal{N}(T)$ and $\mathcal{R}(T)$  the domain, the null space and respectively, the range of an operator $T$. Obviously, we have $B(X, Y) \subset H(X,Y)$.

\begin{defn}
Let $X, \,Y$ be Banach spaces and $M \subset X$ be a subset. Let $T:X \to Y$ be a mapping. Then we called $T$ is
quasi--additive on $M$ if $T$ satisfies
\begin{eqnarray*}
T(x+z)=T(x) +T(z), \qquad \forall\; x\in X, \;\forall\; z\in M.
\end{eqnarray*}
For a homogeneous operator $T\in H(X, X)$, if $T$ is quasi--additive on $\R(T)$, then we will simply call $T$ is a quasi--linear operator.
\end{defn}

\begin{defn}[cf.\cite{WYW1,WLP1}]\label{mdef1.1}
Let $P\in H(X,X)$. If $P^2=P$, we call $P$ is a homogeneous projector. In addition, if $P$ is also quasi--additive on $\R(P)$,
i.e., for any $x \in X$ and any $z\in \R(P)$,
  $$
  P(x + z) = P(x) + P(z) = P(x) + z,
  $$
then we call $P$ is a quasi--linear projector.
\end{defn}

Now we recall the definition of dual mapping for Banach spaces.
\begin{defn}[cf.\cite{BP1}]\label{dualfn1}
Let X be a Banach space. Then the set--valued mapping $F_X : X \to X^*$ defined by $$F_X(x) = \{f \in X^* \mid f(x)=\|x\|^2=\|f\|^2\},\qquad \forall x\in X$$
is called the dual mapping of $X$, where $X^*$ is the dual space of $X$.
\end{defn}

It is well known that the dual mapping $F_X$ is a homogeneous set--valued mapping; $F_X$ is surjective if and only if $X$ is reflexive; $F_X$ is injective or strictly monotone
if and only if $X$ is strictly convex; $F_X$ is single--valued if and only if $X$ is smooth. We will need these properties of $F_X$ to prove our main results for the Moore--Penrose metric generalized inverse. Please see \cite{BP1} for more information about the mapping $F_X$.

\begin{defn}[cf.\cite{IS1}]\label{metricpjdf1}
Let $X$ be a Banach space and $G \subset X$ be a subset of $X$. The set--valued mapping $P_G : X \to G$ defined by
\begin{eqnarray*}
P_G(x) = \{s \in G \mid \|x-s\| = \D(x, G)\},\qquad   \forall x\in X
\end{eqnarray*}
is called the set--valued metric projection, where $\D(x, G) = \inf_{z \in X} \|x-z\|$.
\end{defn}

For a subset $G \subset X$, if $P_G(x) \neq  \emptyset$ for each $x \in X$, then $G$ is said to be approximal; if $P_G(x)$ is at most a singleton
for each $x \in X$, then $G$ is said to be semi--Chebyshev; if $G$ is simultaneously approximal and semi--Chebyshev set, then $G$ is called a Chebyshev set. We denote by $\pi_G$ any selection for the set-valued mapping $P_G$, i.e., any single--valued mapping $\pi_G : \mathcal{D}(\pi_G) \to G$ with the property that $\pi_G(x) \in P_G(x)$ for any $x \in \mathcal{D}(\pi_G)$, where $\mathcal{D}(\pi_G) = \{x \in X : P_G(x) \neq  \emptyset\}$. For the particular case, when $G$ is a Chebyshev set, then $\mathcal{D}(\pi_G) = X$ and $P_G(x) = \{\pi_G(x)\}$. In this case, the mapping $\pi_G$ is called the metric projector from $X$ onto $G$.

\begin{rem}[cf.\cite{IS1}]\label{Chebyshevprp1}
Let $X$ be a Banach space and $G \subset X$ be a closed convex subset. It is well--known that if $X$ is reflexive, then $G$ is a proximal set; if $X$ is a strictly convex, then $G$ is a semi--Chebyshev set. Thus, every closed convex subset in a reflexive strictly convex Banach space is a Chebyshev set.
\end{rem}

The following lemma gives some important properties of the metric projectors.

\begin{lem}[cf.\cite{IS1}]\label{metricpjprp1}
Let $X$ be a Banach space and $L$ be a subspace of $X$. Then
\begin{enumerate}
  \item[\rm(1)]  $\pi_L^2(x) = \pi_L(x)$ for any $x \in \mathcal{D}(\pi_L)$, i.e., $\pi_L$ is idempotent;
  \item[\rm(2)] $\|x-\pi_L(x)\| \leq \|x\|$ for any $x \in \mathcal{D}(\pi_L)$, i.e., $\|\pi_L\|\leq 2$.
\end{enumerate}
In addition, If $L$ is a semi--Chebyshev subspace, then
\begin{enumerate}
  \item[\rm(3)]  $\pi_L(\lambda x) = \lambda \pi_L(x)$ for any $x \in X$ and $\lambda \in \mathbb{R}$, i.e., $\pi_L$ is homogeneous;
  \item[\rm(4)] $\pi_L(x + z) = \pi_L(x) + \pi_L(z) = \pi_L(x) + z$ for any $x \in \mathcal{D}(\pi_L)$ and $z \in L$, i.e., $\pi_L$ is quasi--additive on $L$.
\end{enumerate}
\end{lem}

The following so called generalized orthogonal decomposition theorem in Banach space is a main tool in this paper.

\begin{lem}[Generalized Orthogonal Decomposition Theorem\cite{HWZh1,WYW1}]\label{godtlem1}
Let $X$ be a Banach space and $G\subset X$ be a proximinal subspace. Then for any $x \in X$, we have
\begin{enumerate}
  \item[$(1)$] $x=x_1+x_2$ with $x_1 \in G$ and $x_2 \in F_X^{-1}(G^{\perp})$;
  \item[$(2)$] Furthermore, if $G\subset X$ is a Chebyshev subspace, then the decomposition in (1) is unique such that $x=\pi_G(x)+x_2$. In this case, we can write $X= G\dotplus F_X^{-1}(G^{\perp})$.
\end{enumerate}
Where, $G^{\perp}=\{f\in X^\ast | f(x)=0, \forall x\in G\}$ and  $F_X^{-1}(G^{\perp})=\{x\in X | F_X(x) \cap G^{\perp} \neq \emptyset\}$.
\end{lem}

Now we give the definition of the Moore--Penrose metric generalized for $T\in B(X, Y)$.

\begin{defn}[\cite{WYW1,WWH11}]\label{metricdef1.1}
Let $X$ and $Y$ be Banach spaces. Let $T \in B(X, Y)$. Suppose that $\mathcal{N}(T)$ and $\mathcal{R}(T)$ are Chebyshev subspaces of $X$ and $Y$, respectively. If there exists a bounded homogeneous operator $T^M: Y \to X$ such that:
\begin{eqnarray*}
(1)\, TT^MT = T; \quad (2)\, T^MTT^M= T^M; \quad  (3)\, T^MT = I_X-\pi_{\mathcal{N}(T)}; \quad  (4)\, TT^M= \pi_{\mathcal{R}(T)}.
\end{eqnarray*}
Then $T^M$ is called the Moore--Penrose metric generalized inverse of $T$, where $\pi_{\mathcal{N}(T)}$ and $\pi_{\mathcal{R}(T)}$ are the metric projectors onto $\mathcal{N}(T)$ and $\mathcal{R}(T)$, respectively.
\end{defn}

If $T^M$ exists, then it is also unique(cf. \cite{WYW1,WWH11}). Moreover, if $T^M$ exists, then by Lemma \ref{godtlem1}, the spaces $X$ and $Y$ have the following unique decompositions
\begin{eqnarray*}
X=\mathcal{N}(T) \dotplus F_X^{-1}(\mathcal{N}(T)^\perp),\qquad Y=\mathcal{R}(T) \dotplus F_Y^{-1}(\mathcal{R}(T)^\perp),
\end{eqnarray*}
respectively, where $F_X : X \to X^*$ (resp. $F_Y : Y \to Y^*$) is the set--valued dual mapping of $X$ (resp. $Y$). Please see \cite{WYW1} for more information about the Moore--Penrose metric generalized inverses and related knowledge. Here we only need the following result which characterizes the existence of the Moore--Penrose metric generalized inverse.

\begin{lem}[\cite{NRX1,WYW1}]\label{extlem1.1}
Let $X$ and $Y$ be Banach spaces. Let $T \in B(X, Y)$. If $\mathcal{N}(T)$ and $\mathcal{R}(T)$ are Chebyshev subspaces of $X$ and $Y$, respectively. Then there exists a unique Moore--Penrose metric generalized inverse $T^M$of $T$.
\end{lem}

Finally, in this section, we list some basic results about the reduced minimum module and the gap between two subspaces in a Banach space. For more information, please see \cite{TK,XYF1}.

Let $X$ and $Y$ be Banach spaces, let $T\in B(X, Y)$. The reduced minimum module $\gamma(T)$ of $T$ is defined by
\begin{align*}
\gamma(T)=\inf\{\|Tx\| \mid \D(x, \mathcal{N}(T))=1,\;\; \forall x\in X\}.
\end{align*}

\begin{rem}\label{reducedm1}
From the definition of $\gamma(T)$, it is easy to see that $\|Tx\|\geq \gamma(T)\D(x, \mathcal{N}(T))$ for any $x\in X$. Moreover, according to \cite[Theorem 5.2]{TK}, we know that $\mathcal{R}(T)$ is closed if and only if $\gamma(T)>0$.
\end{rem}

Let $X$ be a Banach space, let $M,\,N$ be two closed subspaces in $X$. We denote by $S_N$ the unit sphere of $N$ (i.e., the set of all $u \in N$ with $\|u\| = 1$). Set
$$
\delta(M,N)=\begin{cases}\sup\{\D(x,N)\,\vert\,x\in M,\,\|x\|=1\},\quad &M\not=\{0\}\\ 0 \quad& M=\{0\}\end{cases}.
$$

\begin{prop}[\cite{TK}]\label{2ssP1}
Let $M,\,N$ be closed subspaces in a Banach space $X$. Then
\begin{enumerate}
  \item[$(1)$] $\delta(M,N)=0$ if and only if $M\subset N$;
  \item[$(2)$] $\hat{\delta}(M,N)=0$ if and only if $M=N$;
  \item[$(3)$] $\hat{\delta}(M,N)=\hat{\delta}(N,M)$;
  \item[$(4)$] $0\leq \delta(M,N)\leq 1$, $0\leq \hat{\delta}(M,N)\leq 1$.
\end{enumerate}
\end{prop}

\section{The simplest expression and existence of the Moore--Penrose metric generalized inverse of the perturbed operator}

In order to proof the main results in the paper, we need the following three lemmas.
\begin{lem}[{\cite[Lemma 3.1, Lemma 3.2, Theorem 3.4]{CX}}]\label{qlem2.8}
Let $T\in B(X,Y)$ such that $T^M$ exists and let $\delta T\in B(X,Y)$ such that $T^M$ is quasi--additive on
$\Ran(\delta T)$ and $\|T^M\delta T\|<1$. Put $\bar T=T+\delta T$. Then
\begin{enumerate}
\item[$(1)$] $I_X+T^M\delta T$ and $I_Y+\delta TT^M$ are invertible in $B(X,X)$ and $H(Y,Y)$, respectively;
\item[$(2)$] $\Phi=(I_X+T^M\delta T)^{-1}T^M=T^M(I_Y+\delta T T^M)^{-1}\in H(Y,X)$;
\item[$(3)$] $\bar T\Phi\bar T=\bar T$ and $\Phi\bar T\Phi=\Phi$ when $\Ran(\bar T)\cap\Ker(T^M)=\{0\}$.
\end{enumerate}
\end{lem}

\begin{lem}\label{qlem11.3}
Let $T\in B(X,Y)$ with $\Ran(T)$ closed and $\mathcal{N}(T)$ (resp. $\mathcal{R}(T)$) Chebyshev in $X$ (resp. $Y$).
Let $\delta T\in B(X,Y)$ with $T^M$ quasi--additive on $\Ran(\delta T)$ and $\|T^M\delta T\|<1$. Put $\bar T=T+\delta T$ and
$\Phi=T^M(I_Y+\delta TT^M)^{-1}=(I_X+T^M\delta T)^{-1}T^M$.
Then $\mathcal{R}(\Phi)=F_X^{-1}(\mathcal{N}(T)^\perp)$ and $\mathcal{N}(\Phi)=F_Y^{-1}(\mathcal{R}(T)^\perp)$.
\end{lem}

\begin{proof}
From Lemma \ref{qlem2.8}, we see $\Phi$ is well defined. Then, according to the expressions of $\Phi$, we have $\Ran(\Phi)=\Ran(T^M)$ and $\Ker(\Phi)=\Ker(T^M)$.
From Lemma \ref{godtlem1} and Definition \ref{metricdef1.1}, we can see that
$\mathcal{R}(T^M)=F_X^{-1}(\mathcal{N}(T)^\perp)$ and $\mathcal{N}(T^M)=F_Y^{-1}(\mathcal{R}(T)^\perp)$.
\end{proof}

For convenience, we recall the concept of smoothness of Banach space. Let $X^*$ be the dual space of $X$. Let $S(X)$ and
$S(X^*)$ be the unit spheres of $X$ and $X^*$, respectively. We say $X$ is smooth if for each point $x\in S(X)$ there
exists a unique $f \in S(X^*)$ such that $f(x)=1$. Please see \cite{Diestel1} for more information about this important
concept and related topics. We have indicated that if $X$ is smooth, then the dual mapping $F_X$
(see Definition \ref{dualfn1}) is single-valued.

\begin{lem}\label{smoothlem1}
Let $M, N \subset X$ be Chebyshev subspaces of $X$. If $X$ is smooth, then $F_X^{-1}(M^{\perp})=F_X^{-1}(N^{\perp})$
if and only if $M=N$.
\end{lem}

\begin{proof}
If $M=N$, obviously, we have $F_X^{-1}(M^{\perp})=F_X^{-1}(N^{\perp})$.

Suppose that $F_X^{-1}(M^{\perp})=F_X^{-1}(N^{\perp})\triangleq G$. We prove that $M =N$ if $X$ is smooth. In fact, since $M, N$ are Chebyshev subspace of $X$, by the Generalized Orthogonal Decomposition Theorem (cf. Lemma \ref{godtlem1}) in Banach space, we have $$X=M \dotplus F_X^{-1}(M^{\perp})=N\dotplus F_X^{-1}(N^{\perp}).$$
Then for any $m\in M \backslash N$, we have the unique decomposition $m=m+0$ with respect to $M$ and $G$. Noting that we
also have the unique decomposition $m=n_1+n_2$ with respect to $N$ and $G$, where $n_1\in N$ and $n_2 \in G$. If
$M \neq N$, by the uniqueness of the decomposition, we must have $n_2\neq 0$.

Since $X$ is smooth, then $F_X$ is single--valued. So from $F_X^{-1}(M^{\perp})=F_X^{-1}(N^{\perp})\triangleq G$,
we get that $f=F_X(n_2)\in N^\perp\cap M^\perp$, that is, $f(n_2)=\|f\|^2=\|n_2\|^2$ and $f(m)=f(n_1)=0$. Therefore,
$\|n_2\|^2=f(m-n_1)=0$ and $n_2=0$, which is a contradiction. Thus, we have $M=N$.
\end{proof}

Under the geometric assumptions that both $X$ and $Y$ are smooth Banach spaces, now we can prove the following useful result for the perturbation of the Moore--Penrose metric generalized inverse of the perturbed operator.

\begin{thm}\label{mainperpthm1}
Let $X, Y$ be smooth Banach spaces and let $T\in B(X,Y)$ with $\Ran(T)$ closed. Let $\delta T\in B(X,Y)$ and put
$\bar{T}=T+\delta T$. Assume that $\mathcal{N}(T)$ and $\mathcal{N}(\bar{T})$ are Chebyshev subspaces of $X$,
$\mathcal{R}(T)$ and $\mathcal{R}(\bar{T})$ are Chebyshev subspaces of $Y$. Then the Moore--Penrose metric generalized
inverse $T^M$ of $T$ exists. In addition, if $T^M$ is quasi--additive on $\mathcal{R}(\delta T)$ and $I_X+T^M\delta T$
is invertible in $B(X,X)$, then $\Phi=T^M(I_Y + \delta T T^M)^{-1}=(I_X + T^M\delta T)^{-1}T^M$ is well defined and
the following statements are equivalent:
\begin{enumerate}
  \item[$(1)$] $\Phi$ is the Moore--Penrose metric generalized inverse of $\bar{T}$, i.e., $\Phi=\bar{T}^M;$
  \item[$(2)$] $\mathcal{R}(\bar{T})=\mathcal{R}(T)$ and $\mathcal{N}(\bar{T})=\mathcal{N}(T);$
  \item[$(3)$] $\mathcal{R}(\delta T) \subset \mathcal{R}(T)$ and $\mathcal{N}(T) \subset \mathcal{N}(\delta T)$.
\end{enumerate}
\end{thm}

\begin{proof}
Since $\mathcal{N}(T)$ and  $\mathcal{R}(T)$ are Chebyshev subspaces of $X$ and $Y$, respectively, then from Lemma
\ref{extlem1.1}, we know $T^M$ exists and is unique. If $T^M$ is quasi--additive on $\mathcal{R}(\delta T)$ and
$I_X+T^M\delta T$ is invertible in $B(X,X)$, then from Lemma \ref{qlem2.8}, we see
$$
\Phi=T^M(I_Y + \delta T T^M)^{-1}=(I_X + T^M\delta T)^{-1}T^M
$$
is well defined. Now we show that the equivalences hold.

$(1)\Rightarrow (2)$ Since $\Phi= \bar{T}^M$, then from Lemma \ref{qlem11.3}, we get that
$$
\mathcal{N}(\bar{T}^M)=\mathcal{N}(\Phi)=F_X^{-1}(\mathcal{N}(T)^\perp),\quad \mathcal{R}(\bar{T}^M)=
\mathcal{R}(\Phi)=F_X^{-1}(\mathcal{R}(T)^\perp).
$$
Since $\mathcal{N}(\bar{T})$ and  $\mathcal{R}(\bar{T})$ are Chebyshev subspaces of $X$ and $Y$, respectively, it follows from Lemma \ref{godtlem1} and Definition \ref{metricdef1.1} that $\mathcal{R}(\bar{T}^M)=
F_X^{-1}(\mathcal{N}(\bar{T})^\perp)$ and $\mathcal{N}(\bar{T}^M)=F_Y^{-1}(\mathcal{R}(\bar{T})^\perp)$.
Consequently,
\begin{align}\label{eqfxfy1}\begin{split}
F_X^{-1}(\mathcal{N}(\bar{T})^\perp)&=\mathcal{N}(\bar{T}^M)=\mathcal{N}(T^M)=F_X^{-1}(\mathcal{N}(T)^\perp),\\
F_Y^{-1}(\mathcal{R}(\bar{T})^\perp)&=\mathcal{R}(\bar{T}^M)=\mathcal{R}(T^M)=F_X^{-1}(\mathcal{R}(T)^\perp).
\end{split}\end{align}
Noting that $X$ and $Y$ are smooth Banach spaces, so we have $\mathcal{R}(\bar{T})=\mathcal{R}(T)$ and
$\mathcal{N}(\bar{T})=\mathcal{N}(T)$ from Lemma \ref{smoothlem1} and \eqref{eqfxfy1}.

$(2)\Rightarrow (3)$  Let $x\in \Ker(T) = \mathcal{N}(\bar{T})$. Then $Tx = 0$ and $Tx + \delta Tx = 0$. So
$\delta Tx = 0$, that is, $\mathcal{N}(T) \subset \mathcal{N}(\delta T)$. Let $y \in \Ran(\delta T)$, then there exists
some $x \in X$ such that $y=\delta Tx=\bar{T}x-Tx$. Noting that $\mathcal{R}(\bar{T})=\mathcal{R}(T)$, we have
$y\in \mathcal{R}(T)$, that is, $\mathcal{R}(\delta T) \subset \mathcal{R}(T)$.

$(3)\Rightarrow (1)$ From $\Ran(\delta T)\subset\Ran(T)$ and $\Ker(T)\subset\Ker(\delta T)$, we get that
$\pi_{\Ran(T)}\delta T=\delta T$ and $\delta T\pi_{\Ker(T)}=0$, that is, $TT^M\delta T=\delta T=\delta TT^MT$. Consequently,
\begin{equation}\label{eqfjdelta1}
\bar T=T+\delta T=T(I_X+T^M\delta T)=(I_Y+\delta TT^M)T.
\end{equation}
Since $I_X+T^M\delta T$ is invertible in $B(X,X)$ and $I_Y+\delta TT^M$ is invertible in $H(Y,Y)$ by Lemma \ref{qlem2.8},
we have $\Ran(\bar T)=\Ran(T)$ and $\Ker(\bar T)=\Ker(T)$ by \eqref{eqfjdelta1}. Thus $\bar T\Phi\bar T=\bar T$ and
$\Phi\bar T\Phi=\Phi$ by Lemma \ref{qlem2.8} and moreover,
\begin{align*}
\bar{T}\Phi&=(T+\delta T)T^M(I_Y + \delta T T^M)^{-1}= TT^M(I_Y+\delta TT^M)(I_Y + \delta T T^M)^{-1}\\&=TT^M=\pi_{\mathcal{R}(T)}=\pi_{\mathcal{R}(\bar{T})};\\
\Phi\bar{T}&=(I_X + T^M\delta T)^{-1}T^M(T+\delta T)= (I_X + T^M\delta T)^{-1}(I_X+T^M\delta T)T^MT\\&=T^MT=I_X-\pi_{\mathcal{N}(T)}=I_X-\pi_{\mathcal{N}(\bar{T})}.
\end{align*}
Therefore, $\Phi$ is the Moore--Penrose metric generalized inverse of $\bar{T}$, i.e., $\Phi= \bar{T}^M$.
\end{proof}

\begin{rem}
We should remark that, some related results of Theorem \ref{mainperpthm1} have been proved in \cite{Mah1}.
In \cite[Theorem 4.3.3]{Mah1}, under the assumptions that
\begin{enumerate}
\item[(1)] $\mathcal{N}(T)$ and $\mathcal{N}(\bar{T})$ are Chebyshev subspaces of $X$;
\item[(2)] $\mathcal{R}(T)$ and $\mathcal{R}(\bar{T})$ are Chebyshev subspaces of $Y$;
\item[(3)] $\|T^M \delta T\|<1$, $\mathcal{N}(T) \subset \mathcal{N}(\delta T)$ and $\mathcal{R}(\delta T)
 \subset \mathcal{R}(T)$;
\item[(4)] $F_X^{-1}(\mathcal{N}(T)^\perp)$ is a linear subspace of $X$ and $\R(T)$ is approximatively compact,
\end{enumerate}
the author proved that $\bar{T}^M$ exists and has the representations
$$
\bar{T}^M=T^M(I_Y + \delta T T^M)^{-1}=(I_X + T^M\delta T)^{-1}T^M.
$$
Thus, Theorem \ref{mainperpthm1} gives some generalization of the above results. We also note that our proof is more
concise. Please see \cite{Mah1} for more related results.
\end{rem}

From Theorem \ref{mainperpthm1}, it is easy to get the following perturbation result for the Moore--Penrose orthogonal projection generalized inverses of bounded linear operators in Hilbert spaces.

\begin{cor}[{\cite[Theorem 3.1]{DJ12}}]\label{djthm1.1}
Let $H, K$ be Hilbert spaces. Let $T \in B(H, K)$ have the Moore--Penrose generalized inverse $T^\dagger\in B(K, H)$.
Let $\delta T \in  B(H, K)$ with $\|T^\dagger \delta T\|<1$. Then $G=(I_X+T^\dagger \delta T)^{-1}T^\dagger$ is the
Moore--Penrose generalized inverse of $\bar{T}= T+\delta T$ if and only if $\mathcal{R}(\bar{T})=\mathcal{R}(T)$ and
$\mathcal{N}(\bar{T})=\mathcal{N}(T).$
\end{cor}

\begin{proof}
Since $H$ and $K$ are Hilbert spaces, then from Definition \ref{metricpjdf1}, we see that the metric projector is just
the linear orthogonal projector. Now from Definition \ref{metricdef1.1}, we see obviously that the Moore--Penrose
metric generalized inverse $T^M$ of $T$ is indeed the Moore--Penrose orthogonal projection generalized inverse
$T^\dagger$ of $T$ under usual sense. It is well--known that Hilbert spaces are smooth and the condition
$\|T^\dagger \delta T\|<1$ implies $I_X+T^\dagger \delta T$ invertible, hence we can get the assertion by using
Theorem \ref{mainperpthm1}.
\end{proof}

Finally, in this section, by using the reduced minimum module and the gap function, we will give two simple results related to the existence of the Moore--Penrose metric generalized inverse of the perturbed operator in reflexive strictly convex Banach spaces. The following lemma is taken from \cite{XYF1}, which is proved for densely defined closed linear operators in Banach spaces. For our purpose, here we present it for bounded linear operators.

\begin{lem}[cf.\cite{XYF1}]\label{lammu1}
Let $X, Y$ be Banach spaces and $T, S \in B(X, Y)$. Suppose that there exist two constants $\lambda>0$ and $\mu\in \mathbb{R}$ such that $\|S x\|\geq \lambda\|Tx\|+\mu\|x\|$ for any $x\in X$, then
\begin{eqnarray*}
\gamma(S)\geq \lambda \gamma(T)(1-2\delta(\mathcal{N}(T), \mathcal{N}(S)))+\mu.
\end{eqnarray*}
\end{lem}

\begin{thm}\label{redugapthm1}
Let $X, Y$ be reflexive strictly convex Banach spaces. Let $T, \delta T \in B(X, Y)$ with $\mathcal{R}(T)$ closed. Put $\bar{T}=T+\delta T$. Suppose that
\begin{eqnarray}\label{eqlabmu1}
\gamma(T)^{-1}\|\delta T\|<1\;\; and \; \;\delta(\mathcal{N}(T), \mathcal{N}(\bar{T}))<\dfrac{1}{2}(1-\gamma(T)^{-1}\|\delta T\|).
\end{eqnarray}
Then the Moore--Penrose metric generalized inverse $\bar{T}^M$ of $\bar{T}$ exists.
\end{thm}

\begin{proof}
Since $\|\bar{T}x\|\geq \|Tx\|-\|\delta T\|\|x\|$ for any $x\in X$, thus we can choose $S=\bar{T}$, $\lambda=1$ and $\mu=-\|\delta T\|$ in Lemma \ref{lammu1}. Now by using \eqref{eqlabmu1}, we can compute
\begin{align*}
\gamma(\bar{T})&\geq \gamma(T)(1-2\delta(\mathcal{N}(T), \mathcal{N}(\bar{T})))-\|\delta T\|\\&= \gamma(T)-2\gamma(T)\delta(\mathcal{N}(T), \mathcal{N}(\bar{T}))-\|\delta T\|\\&> \gamma(T)-\gamma(T)(1-\gamma(T)^{-1}\|\delta T\|)-\|\delta T\|\\&=0.
\end{align*}
Thus from Remark \ref{reducedm1}, we get $\mathcal{R}(\bar{T})$ is closed. $\mathcal{N}(\bar{T})$ is closed since $\bar{T} \in B(X, Y)$. Noting that $X, Y$ are reflexive strictly convex Banach spaces, then by Remark \ref{Chebyshevprp1}, we get that $\mathcal{N}(\bar{T})$ and $\mathcal{R}(\bar{T})$ are Chebyshev subspaces of $X$ and $Y$, respectively. From Lemma \ref{extlem1.1}, we see $\bar{T}^M$ uniquely exists. This completes the proof.
\end{proof}

\begin{cor}
Let $X, Y$ be reflexive strictly convex Banach spaces. Let $T, \delta T \in B(X, Y)$ be such that the Moore--Penrose metric generalized inverse $T^M$ of $T$ exists. Put $\bar{T}=T+\delta T$. Suppose that
\begin{eqnarray*}
\|T^M\|\|\delta T\|<1\;\; and \; \;\delta(\mathcal{N}(T), \mathcal{N}(\bar{T}))<\dfrac{1}{2}(1-\|T^M\|\|\delta T\|).
\end{eqnarray*}
Then the Moore--Penrose metric generalized inverse $\bar{T}^M$ of $\bar{T}$ exists.
\end{cor}

\begin{proof}
Thanks to Theorem \ref{redugapthm1}, we can prove our result by showing that $\|T^M\|\geq\gamma(T)^{-1}$. But this inequality follows from \cite[Lemma 4.1.1]{Mah1}.
\end{proof}

\section{Stability of some operator equations in Banach spaces}
In this section, by using our main perturbation results of the Moore--Penrose metric generalized inverses, we will study the stability of the solutions of the operator equation $Tx = y$ and the best approximate solutions(BAS) of the operator equation $\|Tx-b\|=\inf_{y\in X}\|Ty-b\|$ under different conditions. Throughout this section, we always assume that $X$ and $Y$ are Banach spaces, we also assume that $T \in(X, Y)$ such that both $\mathcal{N}(T)$ and $\mathcal{R}(T)$ are Chebyshev subspaces of $X$ and $Y$, respectively, so that the corresponding Moore--Penrose metric generalized inverse $T^M$ of $T$ is well defined as a bounded homogeneous operator. We also let $\delta T\in B(X,Y)$ such that $T^M$ is quasi--additive on $\mathcal{R}(\delta T)$ satisfying $\|T^M\delta T\|<1$ in this section, so that $I_X+T^M\delta T$ and $I_Y+\delta TT^M$ are invertible in $B(X,X)$ and $H(Y,Y)$, respectively.

($\mathbf{i}$) We first consider the following operator equation:
\begin{eqnarray}\label{oqeqsp1}
Tx=b.
\end{eqnarray}
Suppose that the equation \eqref{oqeqsp1} is perturbed to the following consistent operator equation:
\begin{eqnarray}\label{oqeqsp2j1}
\bar{T}z=b.
\end{eqnarray}
Where $\bar{T}=T+\delta T$ and $b\in \mathcal{R}(T)\cap \mathcal{R}(\bar{T})$ with $b\neq 0$. Denoted by $S(T, b)$ and $S(\bar{T}, \bar{b})$ the solution sets of the equations \eqref{oqeqsp1} and \eqref{oqeqsp2j1}, respectively. It is well known that the solution set $S(T, b)$ of the equation \eqref{oqeqsp1} can be written as $T^Mb + \mathcal{N}(T)$. Let $\kappa=\|T^M\|\|T\|$ be the condition number of $T$. Put $\epsilon_b=\dfrac{\|\delta b\|}{\|b\|}$ and $\epsilon_T=\dfrac{\|\delta T\|}{\|T\|}$, we will keep these notations throughout this section. Now we can prove the following theorem on the perturbations of the consistent linear operator equations \eqref{oqeqsp1} and \eqref{oqeqsp2j1} by using the Moore-Penrose metric generalized inverse.

\begin{thm}\label{opeqthm1fj1}
Let $T,\, \delta T\in B(X,Y)$. If $T^M$ is quasi--additive on $\mathcal{R}(T)$, then
\begin{enumerate}
  \item[$(1)$] For any solution $z \in S(\bar{T}, b)$ of the equation \eqref{oqeqsp2j1} there exists a solution $x\in S(T, b)$ of the equation \eqref{oqeqsp1} such that $\dfrac{\|z-x\|}{\|z\|}\leq \kappa\epsilon_{T}$;
  \item[$(2)$] For any solution $z \in S(\bar{T}, b)$ of the equation \eqref{oqeqsp2j1} there exists a solution $x\in S(T, b)$ of the equation \eqref{oqeqsp1} such that
\begin{eqnarray*}
\dfrac{\|z-x\|}{\|x\|}\leq \dfrac{\kappa\epsilon_{T}}{1-\|T^M \delta T\|}.
\end{eqnarray*}
\end{enumerate}
\end{thm}

\begin{proof}
(1) Let $z \in S(\bar{T}, b)$ be a solution of the equation \eqref{oqeqsp2j1}, taking $x=z+T^M \delta Tz$.  From the equations \eqref{oqeqsp1} and \eqref{oqeqsp2j1}, we see $\delta Tz= Tx-Tz$. We first show that $x\in S(T, b)$. In fact, since $T^M$ is quasi-additive on $\mathcal{R}(T)$, then
\begin{align*}
Tx&=T(z+T^M \delta Tz)=Tz+TT^M\delta Tz\\&=Tz+TT^M(Tx-Tz)\\&=Tz+\delta Tz=b.
\end{align*}
So $x$ is a solution of the equation \eqref{oqeqsp1}. Now we can check the error estimate.
\begin{align*}
\dfrac{\|z-x\|}{\|z\|}=\dfrac{\|T^M \delta Tz\|}{\|z\|}\leq\dfrac{\|T\|\|T^M\| \|\delta T\|\|z\|}{\|T\|\|z\|}= \kappa\epsilon_{T}.
\end{align*}

(2) For any $z \in S(\bar{T}, b)$, we also take $x=z+T^M \delta Tz=(I_X+T^M \delta T)z$. It follows from Lemma \ref{qlem2.8} that $(I_X+T^M\delta T)$ is invertible. Thus, we get $z=(I_X+T^M \delta T)^{-1}x$. Now we can compute as follows
\begin{align*}
\dfrac{\|z-x\|}{\|x\|}&=\dfrac{\|(I_X+T^M \delta T)^{-1}x-x\|}{\|x\|}\\&=\dfrac{\|(I_X+T^M \delta T)^{-1}[x-(I_X+T^M \delta T)x]\|}{\|x\|}\\&\leq\dfrac{\|(I_X+T^M \delta T)^{-1}\|\|T\|\|T^M\|\|\delta T\|\|x\|}{\|T\|\|x\|}\\&\leq\dfrac{\kappa\epsilon_{T}}{1-\|T^M \delta T\|}.
\end{align*}
Thus, we get our results.
\end{proof}

($\mathbf{ii}$) Now suppose that the equation \eqref{oqeqsp1} is perturbed to the following consistent linear operator equation:
\begin{eqnarray}\label{oqeqsp2}
\bar{T}z=\bar{b}.
\end{eqnarray}
Where $\bar{T}=T+\delta T$ and $\bar{b}=b+\delta b\in \mathcal{R}(\bar{T})$. Denoted by $S(\bar{T}, \bar{b})$ the solution set of the equation \eqref{oqeqsp2}.

\begin{thm}\label{opeqthm1}
Let $T,\, \delta T\in B(X,Y)$. If $T^M$ is quasi--additive on $\mathcal{R}(T)$, then for any solution $z \in S(\bar{T}, \bar{b})$ of the equation \eqref{oqeqsp2} there exists a solution $x\in S(T, b)$ of the equation \eqref{oqeqsp1} such that
\begin{eqnarray*}
\dfrac{1}{1+\|T^M\delta T\|}\left(\dfrac{\|T^M\delta b\|}{\|T^Mb\|+2\|z\|}-\kappa\epsilon_T\right)\leq\dfrac{\|z-x\|}{\|x\|}\leq \dfrac{\kappa(\epsilon_b+\epsilon_T)}{1-\|T^M\delta T\|}.
\end{eqnarray*}
\end{thm}

\begin{proof}
Let $z \in \hat{S}(\bar{T}, \bar{b})$ be a solution of the equation \eqref{oqeqsp2}, put $x=T^Mb+(I_X-T^MT)z$. Then, we can check that $x\in S(T, b)$. Noting that $T^M$ is quasi-additive on $\mathcal{R}(T)$, then
\begin{eqnarray}\label{eqna1}
z-x=T^MTz-T^Mb=T^M(Tz-b)\in R(T^M)= F_X^{-1}(\mathcal{N}(T)^\perp).
\end{eqnarray}
It follows that $\pi_{\mathcal{N}(T)}(z-x)=0$ and then $$T^MT(z-x)=(I_X-\pi_{\mathcal{N}(T)})(z-x)=z-x.$$ Now from $\bar{T}z=\bar{b}$ and $Tx=b$, we can check that
\begin{align}\label{eqna2}\begin{split}
(I_X+T^M\delta T)(z-x)&=T^M(T+\delta T)(z-x)\\&=T^M\bar{T}(z-x)=T^M(\bar{T}z-Tx-\delta Tx)\\&=T^M(\delta b-\delta Tx).
\end{split}\end{align}
From Lemma \ref{qlem2.8}, we know that $(I_X+T^M\delta T)$ is invertible. Thus, from \eqref{eqna2}, we get
\begin{align*}
z-x&=(I_X+T^M\delta T)^{-1}T^M(\delta b-\delta Tx).
\end{align*}
So by using above equation, we can obtain
\begin{align}\label{eqna3}
\dfrac{\|z-x\|}{\|x\|}&=\dfrac{\|(I_X+T^M\delta T)^{-1}T^M(\delta b-\delta Tx)\|}{\|x\|}\nonumber\\& \leq \|(I_X+T^M\delta T)^{-1}\|\dfrac{\|T^M(\delta b-\delta Tx)\|}{\|x\|}\nonumber\\&\leq \dfrac{\kappa(\epsilon_b+\epsilon_T)}{1-\|T^M\delta T\|}.
\end{align}
Noting that $\|x\|\leq \|T^Mb\|+\|\pi_{\N(T)}z\|\leq \|T^Mb\|+2\|z\|$, thus, we also have
\begin{align}\label{eqna4}
\dfrac{\|z-x\|}{\|x\|}&\geq \dfrac{\|T^M(\delta b-\delta Tx)\|}{\|I_X+T^M\delta T\|\|x\|}\geq\dfrac{1}{1+\|T^M\delta T\|}\left(\dfrac{\|T^M\delta b\|}{\|T^Mb\|+2\|z\|}-\kappa\epsilon_T\right).
\end{align}
Now, our result follows from \eqref{eqna3} and \eqref{eqna4}. This completes the proof.
\end{proof}

\begin{cor}
Let $T,\, \delta T\in B(X,Y)$. Assume that $\mathcal{R}(\bar{T})$ is a Chebychev subspace in $Y$ and $T^M$ is quasi--additive on $\mathcal{R}(T)$. If $\mathcal{N}(\bar{T})=\mathcal{N}(T)$ and $\R(\bar T)\cap\N(T^M)=\{0\}$, then $\bar{T}^M$ exists. Furthermore, for equations \eqref{oqeqsp1} and \eqref{oqeqsp2} we have
\begin{eqnarray*}
\dfrac{1}{1+\|T^M\delta T\|}\left(\dfrac{\|T^M\delta b\|}{\|T^Mb\|}-\kappa\epsilon_T\right)\leq\dfrac{\|\bar{T}^M\bar{b}-T^Mb\|}{\|T^Mb\|}\leq \dfrac{\kappa(\epsilon_b+\epsilon_T)}{1-\|T^M\delta T\|}.
\end{eqnarray*}
\end{cor}

\begin{proof}
From Lemma \ref{lemma4Tb}, we know $\bar{T}^M$ uniquely exists, and then $\bar{z}=\bar{T}^M\bar{b}$ is a solution of the equation \eqref{oqeqsp2}. Now from our proof of Theorem \ref{opeqthm1}, we see $x=T^Mb+(I_X-T^MT)\bar{T}^M\bar{b}$. Noting that $\mathcal{N}(\bar{T})=\mathcal{N}(T)$, thus
\begin{align*}
x&=T^Mb+(I_X-T^MT)\bar{T}^M\bar{b}=T^Mb+\pi_{\mathcal{N}(T)}\bar{T}^M\bar{b}
\\&=T^Mb+\pi_{\mathcal{N}(\bar{T})}\bar{T}^M\bar{b}\\&=T^Mb.
\end{align*}
So by using Theorem \ref{opeqthm1}, we can obtain our result.
\end{proof}

($\mathbf{iii}$) Finally, we consider the following non consistent  operator equation:
\begin{eqnarray}\label{oqeqsp5}
\|Tx-b\|=\inf_{y\in X}\|Ty-b\|.
\end{eqnarray}
Suppose that the equation \eqref{oqeqsp5} is perturbed to the following non consistent linear operator equation:
\begin{eqnarray}\label{oqeqsp61}
\|\bar{T}z-\bar{b}\|=\inf_{y\in X}\|\bar{T}y-\bar{b}\|.
\end{eqnarray}
where $b, \,\bar{b}=b+\delta b\in Y$ and $\bar{T}=T+\delta T$. Denoted by $\tilde{S}(T, b)$ and $\tilde{S}(\bar{T}, \bar{b})$ the solution sets of the equations \eqref{oqeqsp5} and \eqref{oqeqsp61}, respectively. By \cite[Proposition 2.3.7]{XYF1}, we know that the equation \eqref{oqeqsp5} (resp. \eqref{oqeqsp61}) has solutions and $\tilde{S}(T, \bar{b})$ (resp. $\tilde{S}(T, b)$) is closed and convex when $X$ and $Y$ are reflexive and $\mathcal{R}(T)$ (resp. $\mathcal{R}(\bar{T})$) is closed. Moreover, if the Moore--Penrose metric generalized inverse $T^M$ (resp. $\bar{T}^M$) exists, then from Definition \ref{metricdef1.1} (or cf. \cite[Theorem 3.2]{WWH11}), we see that the vector $x=T^Mb$ (resp. $z=\bar{T}^M\bar{b}$) is not only a solution to the equation \eqref{oqeqsp5} (resp. \eqref{oqeqsp61}), but also the minimal norm approximate solution of \eqref{oqeqsp5} (resp. \eqref{oqeqsp61}) among all the
solutions. In order to using our main Theorem \ref{mainperpthm1}, from now on, we always assume that $X, Y$ are smooth reflexive Banach spaces, we also assume that $\mathcal{N}(T)$ and $\mathcal{N}(\bar{T})$ are Chebyshev subspaces of $X$, $\mathcal{R}(T)$ and $\mathcal{R}(\bar{T})$ are Chebyshev subspaces of $Y$, so that both $T^M$ and $\bar{T}^M$ exist.

\begin{thm}\label{opeqthm12}
Let $T,\, \delta T\in B(X,Y)$. Assume that $\mathcal{N}(\bar{T})=\mathcal{N}(T)$ and $\mathcal{R}(\bar{T})=\mathcal{R}(T)$. If $T^M$ is quasi--additive on $\mathcal{R}(T)$, then for any solution $z\in \tilde{S}(\bar{T}, \bar{b})$ of the equation \eqref{oqeqsp61} there exists a solution $x\in \tilde{S}(T, b)$ of the equation \eqref{oqeqsp5} such that
\begin{eqnarray*}
\dfrac{\|z-x\|}{\|x\|}\leq \dfrac{\kappa}{1-\|T^M\delta T\|}\left(\dfrac{\|\bar{b}\|+\|b\|}{\|\pi_{\mathcal{R}(T)}b\|}+\epsilon_T\right).
\end{eqnarray*}
\end{thm}

\begin{proof}
The proof is similar to Theorem \ref{opeqthm1}. From our assumption, and by using Theorem \ref{mainperpthm1}, we see $\bar{T}^M$ exists and
$$\bar{T}^M=T^M(I_Y + \delta T T^M)^{-1}=(I_X + T^M\delta T)^{-1}T^M.$$

Now let $z\in \tilde{S}(\bar{T}, \bar{b})$. Noting that $\mathcal{R}(\bar{T})$ is a Chebyshev subspace of $Y$, then we can write $\tilde{S}(\bar{T}, \bar{b})=\bar{T}^M\bar{b}+\mathcal{N}(\bar{T})$. Therefore, we can write $z=\bar{T}^M\bar{b}+s$ with $s \in \mathcal{N}(\bar{T})=\mathcal{N}(T)$.
Put $x=T^Mb+(I_X-T^MT)z$. Then, as in Theorem \ref{opeqthm1}, we can prove that $z-x=T^MT(z-x)$. By using $z=\bar{T}^M\bar{b}+s=T^M(I_Y + \delta T T^M)^{-1}\bar{b}+s$, and noting that $T^M$ is quasi--additive on $\mathcal{R}(T)$, then we can check that
\begin{align}\label{oqeqsp6}
z-x&=T^MT(z-x)=T^MT(\bar{T}^M\bar{b}+s-T^Mb-(I_X-T^MT)z)\nonumber
\\&=T^MT\bar{T}^M\bar{b}-T^Mb\nonumber\\&=T^MTT^M(I_Y + \delta T T^M)^{-1}\bar{b}-T^Mb\nonumber\\&=(I_X +  T^M\delta T)^{-1} T^M\bar{b}-T^Mb\nonumber\\&=(I_X + T^M\delta T)^{-1}(T^M\bar{b}-T^Mb-T^M\delta TT^Mb).
\end{align}
Since $\mathcal{N}(\bar{T})=\mathcal{N}(T)$ implies $\mathcal{N}(T)\subset \mathcal{N}(\delta T)$, thus we have
$$\delta Tx=\delta TT^Mb+\delta T(I_X-T^MT)z=\delta TT^Mb+\delta T\pi_{\mathcal{N}(T)}(z)=\delta TT^Mb.$$
Therefore, by using \eqref{oqeqsp6}, we get
$$z-x=(I_X + T^M\delta T)^{-1}(T^M\bar{b}-T^Mb-T^M\delta Tx).$$
Noting that $Tx=TT^Mb=\pi_{\mathcal{R}(T)}b$, we can get
\begin{align}\label{eqna32}
\dfrac{\|z-x\|}{\|x\|}&=\dfrac{\|(I_X + T^M\delta T)^{-1}(T^M\bar{b}-T^Mb-T^M\delta Tx)\|}{\|x\|}\nonumber\\& \leq \|(I_X + T^M\delta T)^{-1}\|\dfrac{\|T\|\|T^M\|(\|\bar{b}\|+\|b\|+\|\delta T\|\|x\|)}{\|T\|\|x\|}\nonumber\\&\leq \dfrac{\kappa}{1-\|T^M\delta T\|}\left(\dfrac{\|\bar{b}\|+\|b\|}{\|\pi_{\mathcal{R}(T)}b\|}+\epsilon_T\right).
\end{align}
Now, our result follows from \eqref{eqna32}. This completes the proof.
\end{proof}
Noting that, from the proof of Theorem \ref{opeqthm12}, if $b\in \mathcal{R}(T)$, then we can get the same error estimate as in Theorem \ref{opeqthm1}. We also have the following result about the the minimal norm approximate solutions of the equations \eqref{oqeqsp5} and \eqref{oqeqsp61}.

\begin{cor}
Let $T,\, \delta T\in B(X,Y)$. Assume that $\mathcal{N}(\bar{T})=\mathcal{N}(T)$ and $\mathcal{R}(\bar{T})=\mathcal{R}(T)$. If $T^M$ is quasi--additive on $\mathcal{R}(T)$, then for the equations \eqref{oqeqsp5} and \eqref{oqeqsp61} we have
\begin{eqnarray*}
\dfrac{\|\bar{T}^M\bar{b}-T^Mb\|}{\|T^Mb\|}\leq\dfrac{\kappa}{1-\|T^M\delta T\|}\left(\dfrac{\|\bar{b}\|+\|b\|}{\|\pi_{\mathcal{R}(T)}b\|}+\epsilon_T\right).
\end{eqnarray*}
\end{cor}

\begin{proof}
If we take $z=\bar{T}^M\bar{b}$ in our proof of Theorem \ref{opeqthm12}, then we see $x=T^Mb$ since $\mathcal{N}(\bar{T})=\mathcal{N}(T)$. Now our result follows from Theorem \ref{opeqthm12}.
\end{proof}

%----------------------------------------------------------------------------------------
%------------------------------------------------------------------------------------%

\end{document}